\documentclass[12pt]{amsart}
\usepackage[matrix,arrow]{xy}
\usepackage{amssymb}
\usepackage{amsmath}
\usepackage{amsfonts}
\usepackage{epsfig}
\usepackage{color}
\usepackage{epstopdf}
\usepackage{graphicx}
\usepackage{amsthm}
\usepackage{mathtools}
\usepackage{enumerate}
\usepackage[mathscr]{eucal}
\usepackage{verbatim}

\usepackage[bookmarks=true,hyperindex,pdftex,colorlinks,citecolor=blue,
linkcolor=blue,urlcolor=blue]{hyperref}
\usepackage{tikz}
\usetikzlibrary{matrix,arrows.meta}

\theoremstyle{plain}
\newtheorem{theorem}{Theorem}

\newtheorem{prop}[theorem]{Proposition}
\newtheorem{lemma}[theorem]{Lemma}

\theoremstyle{definition}
\newtheorem{example}[theorem]{Example}

\newtheorem*{question}{Question}

\newtheorem{definition}[theorem]{Definition}

\newcommand{\R}{\mathbb{R}}
\newcommand{\N}{\mathbb{N}}

\newcommand{\Lin}{\mathcal{L}}

\newcommand{\F}{\mathcal{F}}
\newcommand{\eps}{\varepsilon}
\newcommand{\lam}{\lambda}
\newcommand{\ph}{\varphi}

\DeclareMathOperator{\spann}{span}

\DeclareMathOperator{\NA}{NA}

\DeclareMathOperator{\MA}{MA}

\DeclareMathOperator{\SNA}{SNA}
\DeclareMathOperator{\SMA}{SMA}

\newcommand{\Mol}{\operatorname{Mol}}

\newcommand{\Lip}{{\mathrm{Lip}}_0}

\title[A bi-Lipschitz characterization of strong minimum-attainment]
{A bi-Lipschitz characterization of strong minimum-attainment for Lipschitz maps}

\author[G.~Choi]{Geunsu Choi}
\address[G.~Choi]{Department of Mathematics Education, Sunchon National University, Jeonnam 57922, Republic of Korea}
\email{\texttt{gschoi@scnu.ac.kr}}


\date{\today}                                           

\begin{document}

\begin{abstract}
 We completely characterize the denseness of strongly minimum-attaining Lipschitz functions, a minimum analogue for strongly norm-attaining Lipschitz functions, in terms of bi-Lipschitz embeddings. More precisely, our main result shows that the set of strongly minimum-attaining Lipschitz functions defined on a complete metric space $M$ fails the denseness if and only if $M$ is bi-Lipschitz equivalent to a subset of $\mathbb{R}$ with positive Lebesgue measure, or equivalently, if $M$ admits a bi-Lipschitz embedding into $\mathbb{R}$ and $M$ has positive 1-dimensional Hausdorff measure. As a consequence, we provide an isometric characterization of the pure 1-unrectifiability of $M$ in terms of strongly minimum-attaining Lipschitz maps defined on bi-Lipschitz copies of closed subsets of $M$. Several counterexamples showing that the main result cannot be naturally extended to the vector-valued setting are also presented.
\end{abstract}

\maketitle

\section{Introduction}

The Bishop-Phelps theorem \cite{BP} initiated one of the central approximation problems in Banach space theory. It states that every bounded linear functional can be approximated arbitrarily well by functionals attaining their norm, and thus raises the natural question of whether the same phenomenon persists for bounded linear operators; given Banach spaces $X$ and $Y$ with $S_X$ the unit sphere of $X$, whether the set of \textit{norm-attaining operators}
$$
\NA(X,Y) = \bigl\{ T \in \Lin(X,Y) : \|T(x)\|=\|T\| \text{ for some } x\in S_X \bigr\}.
$$
is dense or not in the space $\Lin(X,Y)$ of all bounded linear operators from $X$ into $Y$. In contrast with the scalar-valued case, it was first shown by Lindenstrauss \cite{L} that $\NA(X,Y)$ need not be dense in $\mathcal L(X,Y)$ in general. The norm-attaining theory developed around this problem has uncovered a close connection between norm-attainment and the geometry of the domain and range spaces (see \cite{B} for instance). Nevertheless, a complete characterization of the pairs $(X,Y)$ for which $\NA(X,Y)$ is dense in $\Lin(X,Y)$ remains unknown.

A natural Lipschitz analogue of this problem is obtained by replacing the unit sphere of a Banach space with pairs of distinct points of a metric space. Given a (pointed) metric space $M$ and a real Banach space $Y$, we denote by $\Lip(M,Y)$ the Banach space of all real-valued Lipschitz maps from $M$ into $Y$ which vanish at the distinguished point $0$, endowed with the Lipschitz norm
$$
\|f\| = \!\sup_{p \neq q \in M}\! \frac{\|f(p)-f(q)\|}{d(p,q)}.
$$
A Lipschitz map $f\in\Lip(M,Y)$ is said to \textit{strongly attain its norm} if there are points $p \neq q\in M$ such that
$$
\frac{\|f(p)-f(q)\|}{d(p,q)} = \|f\|.
$$
The set of all such Lipschitz maps is denoted by $\SNA(M,Y)$. When $Y = \R$, we simply call $f$ a Lipschitz function and write $\Lip(M)$ and $\SNA(M)$ respectively for convenience. The word ``strongly'' distinguishes this notion from ordinary norm-attainment after identifying $\Lip(M)$ with the dual of the Lipschitz-free space $\F(M)$. Recall that the \textit{Lipschitz-free space} over $M$ is the Banach space
$$
\F(M) := \overline{\spann}\{\delta_x : x \in M\} \subseteq \Lip(M)^*,
$$
where $\delta_x : \Lip(M) \to \R$ is the evaluation functional given by $\delta_x(f) = f(x)$. We define the set of \textit{molecules} by
$$
\Mol(M) = \left\{m_{p,q} := \dfrac{\delta_p - \delta_q}{d(p,q)} : p \neq q \in M\right\} \subseteq S_{\F(M)}.
$$
It is well known that every $f \in \Lip(M,Y)$ corresponds isometrically to an operator $T_f \in \Lin(\F(M),Y)$, determined by $T_f(\delta_x) = f(x)$ for $x \in M$. Thus, from this viewpoint, strong norm attainment means that $T_f$ attains its norm at a molecule, rather than at an arbitrary element of the unit sphere of $\F(M)$.

The denseness of $\SNA(M,Y)$ has been studied intensively in recent years (see \cite{CGMR,KMS} for instance). For example, if $\F(M)$ has the Radon-Nikod\'ym property (\textup{RNP} for short), then the set of strongly norm-attaining Lipschitz maps is dense in $\Lip(M,Y)$ for every Banach space $Y$; similar conclusions follow from several related geometric structures of $\F(M)$ \cite{CCGMR}. On the other hand, there are broad classes of metric spaces for which densness fails. This is the case, for instance, for length spaces or closed subsets of $\R$ with positive Lebesgue measure \cite{CCGMR}, and for $C^1$ curves \cite{Chi}. These results reveal a close connection between strong norm-attainment and the geometry of the underlying metric space. Nevertheless, no complete geometric characterization of those metric spaces $M$ for which $\SNA(M,Y)$ is dense in $\Lip(M,Y)$ is currently known, even when $Y=\R$.

In parallel with the theory of norm attainment, the minimum modulus of bounded linear operators has also attracted increasing attention (see \cite{Cha,CL} for instance). For $T\in\mathcal L(X,Y)$, the \textit{minimum modulus} of $T$ is defined by
$$
m(T) = \inf_{x\in S_X}\|T(x)\|.
$$
We say that $T$ \textit{attains its minimum} if there exists $x\in S_X$ such that $\|T(x)\| = m(T)$, and we denote the set of all minimum-attaining operators from $X$ into $Y$ by $\MA(X,Y)$. Although every non-injective operator trivially attains its vanishing minimum, the approximation problem for minimum-attaining operators is nontrivial, particularly within the class of operators which are bounded below.

Recent work \cite{GMMR} has shown that the densness of minimum-attaining operators is also closely governed by the Banach space geometry in terms of the \textup{RNP}. Their results demonstrate that the minimum-attainment, just as the norm-attainment, reflects subtle geometric properties of Banach spaces. A complete characterization of the individual pairs $(X,Y)$ for which $\MA(X,Y)$ is dense in $\Lin(X,Y)$, however, is still unknown. \\

The purpose of the present paper is to develop a Lipschitz counterpart of this minimum-attainment theory. Given $f \in \Lip(M,Y)$, we define the \textit{minimum modulus} of $f$ by
$$
m(f) = \!\inf_{p \neq q \in M}\! \frac{\|f(p)-f(q)\|}{d(p,q)}.
$$
We say that $f$ \textit{strongly attains its minimum} if there are points $p \neq q \in M$ for which
$$
\frac{\|f(p)-f(q)\|}{d(p,q)} = m(f).
$$
The set of all strongly minimum-attaining Lipschitz maps from $M$ into $Y$ will be denoted by $\SMA(M,Y)$, and by $\SMA(M)$ when $Y=\R$. This turns out to become a very natural definition, as it is clear that a bounded linear operator on a Banach space attains its minimum if and only if it strongly attains its minimum as a Lipschitz map. Since the strong minimum-attainment arises from the aforementioned concepts, our fundamental question is therefore the following:
\begin{question}
For which metric spaces $M$ is $\SMA(M)$ dense in $\Lip(M)$?
\end{question}
The present paper provides an affirmative answer to this question by giving necessary and sufficient conditions on $M$, in contrast to the situation for the strong norm-attainment. In fact, this question is naturally related to the bi-Lipschitz geometry of $M$. A Lipschitz map $f \in \Lip(M,Y)$ is called a \emph{bi-Lipschitz embedding} \cite{O} from $M$ into $Y$ if there exist constants $C>c>0$ such that
$$
c d(p,q) \leq \|f(p)-f(q)\| \leq C d(p,q) \quad \text{for all } p, q \in M.
$$
In this case, we say $M$ is \textit{bi-Lipschitz equivalent} to the image $f(M)$ in $Y$. We will see through Lemma \ref{lemma:bi-Lipschitz} that the density problem is closely connected to the approximation of bi-Lipschitz embeddings from $M$ into $\R$ by minimum-attaining Lipschitz functions, and hence to the 1-dimensional bi-Lipschitz geometry of $M$. Our first main result is the following.

\begin{theorem}\label{theorem:SMA(M)}
Let $M$ be a complete metric space. Then, the following are equivalent.
\begin{enumerate}
\item[\textup{(a)}] $\SMA(M)$ is not dense in $\Lip(M)$.
\item[\textup{(b)}] $M$ is bi-Lipschitz equivalent to a subset of $\R$ with positive Lebesgue measure.
\item[\textup{(c)}] $M$ admits a bi-Lipschitz embedding into $\R$ and $\mathcal{H}^1(M)>0$.
\end{enumerate}
\end{theorem}

Here, $\mathcal{H}^1$ denotes the 1-dimensional Hausdorff measure; for a measurable subset $E$ of $\R$, one has $\mathcal{H}^1(E)=\lambda(E)$, where $\lam$ denotes the Lebesgue measure. Consequently, the failure of densness is completely rigid; it occurs precisely for those complete metric spaces which are, up to bi-Lipschitz equivalence, subsets with positive measure of the real line. Thus, unlike the corresponding minimum-attainment and strong norm-attainment problems, the strong minimum-attainment problem admits an exact geometric characterization for scalar cases.

Recall that a metric space $M$ is called \textit{purely 1-unrectifiable} if every Lipschitz image in $M$ of a subset $\R$ has zero 1-dimensional Hausdorff meausre. In this sense, note that $\mathcal{H}^1(M)>0$ in Theorem \ref{theorem:SMA(M)}.(c) can be replaced by that $M$ is not purely 1-unrectifiable. Also, recall from \cite{AGPP} that $M$ is purely 1-unrectifiable if and only $\F(M)$ has the \textup{RNP}. Let us comment that \textup{RNP} plays an important role in the geometry of Banach spaces, and many results regarding norm- and minimum-attaining operators have already characterized by the \textup{RNP} as aforementioned. Our second main theorem, illustrated as follows, fully characterizes the pure 1-unrectifiability in terms of strongly minimum-attaining Lipschitz maps on bi-Lipschitz copies of closed subsets.

\begin{theorem}\label{theorem:RNP-characterization}
Let $M$ be a complete metric space. Then, the following are equivalent.
\begin{enumerate}
\item[\textup{(a)}] $M$ is purely 1-unrectifiable.
\item[\textup{(b)}] $\SMA(N)$ is dense in $\Lip(N)$ for every bi-Lipschitz copy $N$ of a closed subset of $M$.
\item[\textup{(c)}] $\SMA(N,Y)$ is dense in $\Lip(N,Y)$ for every bi-Lipschitz copy $N$ of a closed subset of $M$ and for every Banach space $Y$.
\end{enumerate}
\end{theorem}

This result should be compared with \cite[Theorem 2.12]{CJ}. Interestingly, while the density problem for $\SMA(M)$ admits a complete characterization, unlike that for $\SNA(M)$, the corresponding isomorphic characterization is nevertheless the same. \\

The present paper is organized as follows. In Section \ref{section:main-theorem}, we give a complete proof of Theorem \ref{theorem:SMA(M)}. To this end, we first show in Proposition \ref{prop:minimum-zero} that every Lipschitz function with vanishing minimum modulus belongs to the closure of $\SMA(M)$. This key result allows us to handle the delicate cases with ease, thereby reducing the problem to bi-Lipschitz embeddings. The next main ingredient is Lemma \ref{lemma:consecutive}, which shows that, when the image of a bi-Lipschitz embedding has Lebesgue measure zero, the minimum modulus of the embedding can be computed by taking the infimum of the Lipschitz slopes only over pairs of points that are consecutive in the sense of Definition \ref{definition:consecutive}. Next, we construct a measurable set with the property that every subset of $\mathbb{R}$ having positive measure in some interval meets both this set and its complement in sets of positive measure. With the aid of these results, we finally give a systematic and constructive proof of the main theorem.

The paper concludes with Section \ref{section:remarks}, where we discuss the vector-valued setting. We first show that if $\F(M)$ has the \textup{RNP}, then the set of strongly minimum-attaining Lipschitz maps is always dense for every range space, which allows us to provide a proof of Theorem \ref{theorem:RNP-characterization}. We then show, through several examples, that the corresponding equivalences of Theorem \ref{theorem:RNP-characterization} for vector-valued version fail in general. We conclude the paper with an open question concerning the characterization of the pairs $(M,Y)$ for which the denseness holds.

\section{Proof of the main theorem}\label{section:main-theorem}

This section is devoted to give a complete proof of Theorem \ref{theorem:SMA(M)}. We start with an easy observation which shows that the bi-Lipschitz embedding is deeply related with the concept of minimum modulus.

\begin{lemma}\label{lemma:bi-Lipschitz}
Let $M$ be a complete metric space and $Y$ be a Banach space. Then for every $f \in \Lip(M,Y)$, $f$ is a bi-Lipschitz embedding from $M$ into $Y$ if and only if $m(f)>0$. In this case, the image $f(M)$ is a closed subset of $Y$.
\end{lemma}

\begin{proof}
The first assertion is clear. For the second assertion, given any Cauchy sequence $(p_n)$ in $M$, the inequality
$$
d(p_n,p_m) \leq \dfrac{1}{m(f)}\|f(p_n)-f(p_m)\| \quad \text{for any } n,m \in \N
$$
and the completeness of $M$ give the conclusion.
\end{proof}

The following key lemma provides the basic approximation result for Lipschitz functions with vanishing minimum modulus, allowing us to focus on bi-Lipschitz embeddings. For later use, we state the vector-valued version of this result.

\begin{prop}\label{prop:minimum-zero}
Let $M$ be a metric space and $Y$ be a Banach space. Then, every Lipschitz map $f \in \Lip(M,Y)$ with $m(f)=0$ can be approximated by strongly minimum-attaining Lipschitz maps.
\end{prop}

\begin{proof}
Let $f \in \Lip(M,Y)$ with $m(f) = 0$ be given. Given any $\eps>0$, there exist 
$u \neq v \in M$ such that
$$
\dfrac{\|f(u)-f(v)\|}{d(u,v)}<\eps.
$$
Define another Lipschitz function $g \in \Lip(M,Y)$ by
$$
g(x) := f(x) + \frac{f(u)-f(v)}{d(u,v)} \cdot ( d(u,x) - d(u,0) ) \quad \text{for } x \in M.
$$
Then, we have that $g$ strongly attains its minimum $m(g) = 0$ at $u, v \in M$ since
$$
\dfrac{g(u)-g(v)}{d(u,v)} = \dfrac{f(u)-f(v)}{d(u,v)} - \dfrac{f(u)-f(v)}{d(u,v)} \cdot \dfrac{d(u,v)}{d(u,v)} =0.
$$
Moreover, observe that
\begin{align*}
\|g-f\| &= \sup_{p \neq q \in M} \dfrac{\|(g-f)(p)-(g-f)(q)\|}{d(p,q)} \\
&= \dfrac{\|f(u)-f(v)\|}{d(u,v)} \sup_{p \neq q \in M} \dfrac{|(d(u,p)-d(u,0))-(d(u,q)-d(u,0))|}{d(p,q)} \\
&\leq \frac{\|f(u)-f(v)\|}{d(u,v)} \cdot \sup_{p \neq q \in M} \dfrac{d(p,q)}{d(p,q)} < \eps,
\end{align*}
which implies that $g$ is the desired Lipschitz map.
\end{proof}

\begin{definition}\label{definition:consecutive}
Let $f \in \Lip(M)$ be a bi-Lipschitz embedding from $M$ into $\R$. Given $p, q \in M$, we will write $p <_f q$ if $f(p)<f(q)$. We say that two points $p \neq q \in M$ are \textit{consecutive with respect to} $f$ if
$$
f(M) \cap (\min\{f(p),f(q)\}, \max \{f(p),f(q)\}) = \emptyset,
$$
and we write $(p,q) \in \mathcal{C}_f(M)$.
\end{definition}

It is clear that $<_f$ defines a total order. The following lemma shows that, in computing the minimum modulus of $f$, it suffices to consider pairs that are consecutive with respect to $<_f$ as mentioned.

\begin{lemma}\label{lemma:consecutive}
Let $M$ be a complete metric space, and let $f \in \Lip(M)$ be a bi-Lipschitz embedding from $M$ into $\R$. If the set $f(M)$ has Lebesgue measure zero, then we have
$$
m(f) = \inf \left\{ \frac{|f(p)-f(q)|}{d(p,q)} : (p,q) \in \mathcal{C}_f(M) \right\}.
$$
\end{lemma}

\begin{proof}
We first show that $\mathcal{C}_f(M)$ is not empty. Fix any $u <_f v \in M$, and write $a := f(u)$, $b := f(v)$. Since $\lam(f(M)) = 0$, it follows that $(a,b) \setminus f(M) \neq \emptyset$ and there exists $s <_f t \in M$ such that
$$
(f(s),f(t)) \cap (a,b) = \emptyset.
$$
Indeed, since $f(M)$ is closed by Lemma \ref{lemma:bi-Lipschitz}, an open set $(a,b) \setminus f(M)$ is a countable union of disjoint open intervals, and endpoints of each interval must lie in $M$. This leads to $(s,t) \in \mathcal{C}_f(M)$.

Now, assume that there exists $\eps>0$ such that
$$
\inf \left\{ \dfrac{|f(p)-f(q)|}{d(p,q)} : (p,q) \in \mathcal{C}_f(M) \right\} \geq m(f) + \eps.
$$
Then, given any $p <_f q \in M$, let us write
$$
(f(p),f(q)) \setminus f(M) =: \bigcup_{j \in I} (f(p_j), f(q_j)).
$$
It follows that
$$
f(q)-f(p) = \sum_{j \in I} (f(q_j) - f(p_j)) \geq (m(f)+\eps) \sum_{j \in I} d(p_j,q_j) \overset{(*)}{\geq} (m(f)+\eps) d(p,q),
$$
which is a contradiction. For $(*)$, as we have
$$
(f(p),f(q)) \setminus f(M) = \bigcup_{j \in I} (f(p_j), f(q_j)),
$$
we can derive for any finitely many choices of $(p_{j_k},q_{j_k})$ for $k=1, \ldots, n$ (assume $p_{j_k} <_f p_{j_{k+1}}$ for simplicity) among $(p_j,q_j)$ that
\begin{align*}
d(p,q) &\leq d(p,p_{j_1}) + \sum_{k=1}^n d(p_{j_k},q_{j_k}) + \sum_{k=1}^{n-1} d(q_{j_k},p_{j_{k+1}}) + d(q_{j_n},q) \\
&\leq \dfrac{1}{m(f)} \left[ (f(q)-f(p)) - \sum_{k=1}^n (f(q_{j_k}) - f(p_{j_k})) \right] + \sum_{k=1}^n d(p_{j_k}, q_{j_k}),
\end{align*}
and the right-hand side converges to $\displaystyle \sum_{j \in I} d(p_j,q_j)$ as the choices are refined.
\end{proof}

The following result shows that every measurable subset of $\R$ contains a measurable subset that splits each of its positive-measure intersections with an interval into two sets of positive measure. We believe that this result is likely known, as its proof is based on a standard rational-endpoint argument. Nevertheless, we include a proof for completeness.

\begin{prop}\label{prop:subset-measure}
Let $E$ be a bounded measurable subset of an interval $I \subseteq \R$ with $\lam(E)>0$. Then, there exists a measurable set $F \subseteq E$ such that whenever an interval $J \subseteq I$ satisfies that $\lam(E \cap J) >0$, we have $0<\lam(F \cap J) < \lam(E \cap J)$.
\end{prop}

\begin{proof}
Set $I = [a,b]$ and consider a non-decreasing absolutely continuous function $\psi : I \to [0,\lam(E)]$ defined by
$$
\psi(t) := \lam([a,t] \cap E) \quad \text{for } t \in [a,b].
$$
We will construct a subset $C \subseteq [0, \lam(E)]$ satisfying $0<\lam(C \cap J) < \lam(J)$ for every open interval $J \subseteq [0, \lam(E)]$.

Let $(J_n)$ be the countable collection of all open intervals in $[0,\lam(E)]$, where the endpoints of each $J_n$ are rational. In $J_1$, fix any two disjoint closed intervals and let $C_1$, $D_1$ be fat Cantor sets with positive measures constructed on each interval. In $J_2$, we may select two closed disjoint intervals which do not intersect with $C_1$ and $D_1$, which is possible since $C_1$ and $D_1$ are nowhere dense. We construct $C_n$ and $D_n$ inductively so that they do not intersect with previous fat Cantor sets, and let $C := \bigcup_{n=1}^\infty C_n$. Then by construction, we have both
\begin{equation}\label{eq:both-positive}
\lam(C \cap J)>0 \quad \text{and} \quad \lam(J \setminus C)>0
\end{equation}
for any open interval $J \subseteq [0, \lam(E)]$ since $J$ contains some $J_k$ which contains $C_k$ with positive measure, whereas it excludes $D_k$ with positive measure.

Now, we define $F := E \cap \psi^{-1}(C)$. We claim that if an interval $J \subseteq I$ satisfies that $\lam(E \cap J)>0$, then $0<\lam(F \cap J) < \lam(E \cap J)$. Given $J = [c,d] \subseteq I$,
\begin{align*}
\lam(F \cap J) = \int_c^d \chi_F(t)\,dt &= \int_c^d \chi_E(t) \cdot \chi_C(\psi(t))\,dt \\
&= \int_c^d \psi'(t) \cdot \chi_C(\psi(t)) \, dt \\
&= \int_{\psi(c)}^{\psi(d)} \chi_C(s) \, ds = \lam(C \cap [\psi(c), \psi(d)]),
\end{align*}
since $\psi$ is absolutely continuous and $\chi_E(t) = \psi'(t)$
almost everywhere. So if $\lam(E \cap J) >0$, then it follows that $\lam([\psi(c), \psi(d)])>0$ otherwise
$$
\lam([a,c] \cap E) = \lam([a,d] \cap E),
$$
which contradicts $\lam(E \cap J)>0$. By construction of $C$, this gives that
$$
0 < \lam(C \cap [\psi(c),\psi(d)]) = \lam(F \cap J) <\lam([\psi(c),\psi(d)]) = \lam(E \cap J)
$$
from \eqref{eq:both-positive}, as desired.
\end{proof}

A useful consequence of the Lebesgue density theorem is that one can find, near a Lebesgue density point, two points of the set whose distance is arbitrarily large relative to the measure of the complement lying between them.

\begin{lemma}\label{lemma:densness-point}
Let $E$ be a measurable subset of $\R$ with $\lam(E)>0$, and let $t$ be a Lebesgue density point of $E \subseteq \R$. Then for every $\eps>0$ and $\delta>0$, there exist $a,b \in E \cap(t-\delta,t+\delta)$ with $a<b$ such that $\eps(b-a) > \lam((a,b) \setminus E)$.
\end{lemma}

\begin{proof}
Let $\eps>0$ and $\delta>0$ be given. Since $t \in E$ is a Lebesgue density point of $E$, there exists $0<\delta_1<\delta$ such that
$$
\min \left\{ \dfrac{\eps}{2}, \dfrac{1}{2} \right\} \cdot 2\delta_1 > \lam((t-\delta_1,t+\delta_1) \setminus E).
$$
Choose any $a \in (t-\delta_1, t-\frac{\delta_1}{2}) \cap E$ and $b \in (t+\frac{\delta_1}{2},t+\delta_1) \cap E$, otherwise $\delta_1 > \lam((t-\delta_1,t+\delta_1) \setminus E) \geq \delta_1$, which is a contradiction. Thus, we have
$$
\lam((a,b) \setminus E) \leq \lam((t-\delta_1,t+\delta_1) \setminus E) \leq \min \left\{ \dfrac{\eps}{2}, \dfrac{1}{2} \right\} \cdot 2\delta_1 \leq \eps \delta_1 \leq \eps (b-a),
$$
which finishes the proof.
\end{proof}

As our final lemma, we record a simple but useful observation; if a Lipschitz function strongly attains its minimum at a pair of points, then every point whose image lies between the corresponding function values gives rise to new pairs at which the minimum is also strongly attained.

\begin{lemma}\label{lemma:minimum-preserve}
Let $M$ be a metric space, and let $f \in \Lip(M)$ be given. If $f$ strongly attains its minimum at some $p \neq q \in M$, then $f$ also strongly attains its minimum both at $p \neq r \in M$ and $r \neq q \in M$ for any $r \in M$ with $f(p)<f(r)<f(q)$.
\end{lemma}

\begin{proof}
Assume that $\dfrac{f(q)-f(p)}{d(p,q)} = m(f)$. Then, we have
\begin{align*}
m(f) d(p,q) = f(q)-f(p) &= [f(q)-f(r)]+ [f(r)-f(p)] \\
&\geq m(f) d(r,q) + m(f) d(p,r) \\
&\geq m(f) d(p,q),
\end{align*}
which leads to that all inequalities are, in fact, equalities. Therefore, it follows that
$$\dfrac{f(r)-f(p)}{d(p,r)} = \dfrac{f(q)-f(r)}{d(r,q)} = m(f),
$$
proving the claim.
\end{proof}

We are now ready to prove the main theorem of this section, which gives a complete characterization of those metric spaces $M$ for which $\SMA(M)$ is dense in $\Lip(M)$.

\begin{proof}[Proof of Theorem 1]
(a)$\,\Longrightarrow\,$(b); If $M$ is not bi-Lipschitz equivalent to any subset of $\R$, then it follows from Lemma \ref{lemma:bi-Lipschitz} that $m(f)=0$ for any $f \in \Lip(M)$. In this case, Proposition \ref{prop:minimum-zero} gives the conclusion. Thus it suffices to show that $\SMA(M)$ is dense in $\Lip(M)$ if $M$ is bi-Lipschitz equivalent to a subset of $\R$ with zero Lebesgue meausre. Let $f \in \Lip(M)$ be a bi-Lipschitz embedding from $M$ into $\R$ such that $\lam(f(M))=0$. Let $\eps>0$ be given with $0<\eps < m(f)$. By Lemma \ref{lemma:consecutive}, there exist $u <_f v \in M$ with $(u,v) \in \mathcal{C}_f(M)$ and $0 \leq \delta \leq \eps$ such that
$$
\dfrac{f(v)-f(u)}{d(u,v)} = m(f) \left(1 + \dfrac{\delta}{\|f\|}\right) < m(f) \left( 1 + \dfrac{\eps}{\|f\|}\right).
$$
Since $<_f$ is of total order, we can define $g \in \Lip(M)$ (may assume $0<_f u$) by
$$
g(x) = \begin{cases}
\,f(x) \,, & \text{if } x <_f u\\
\,f(x) - \gamma \,,  & \text{if } x >_f v,
\end{cases}
$$
where $\gamma := \dfrac{\eps ((f(v)-f(u))}{\|f\|}$. We claim that $\|g-f\| \leq \eps$ and $g \in \SMA(M)$. \\
To see that $\|g-f\| \leq \eps$, it suffices to consider among $p <_f q \in M$ such that $p<_f u<_f v<_f q$, as the other cases are clear. In that case,
\begin{align*}
\dfrac{|(g-f)(q)-(g-f)(p)|}{d(p,q)} = \dfrac{\gamma}{d(p,q)} \leq \dfrac{\gamma \|f\|}{f(q)-f(p)} \leq \dfrac{\gamma \|f\|}{f(v)-f(u)} = \eps.
\end{align*}
To show that $g \in \SMA(M)$, we first have
\begin{equation}\label{eq:smaller-than-m(f)}
\dfrac{|g(u)-g(v)|}{d(u,v)} = \dfrac{f(v)-f(u)-\gamma}{d(u,v)} = m(f) \left( 1+ \dfrac{\delta}{\|f\|} \right) - \dfrac{\gamma}{d(u,v)} \leq m(f)
\end{equation}
from the choice of $\gamma$. It is easy to see that
$$
m(f) \left( 1 + \dfrac{\delta}{\|f\|}\right) - \dfrac{\gamma}{d(u,v)} \geq 0
$$
since $\dfrac{\gamma}{d(u,v)} \leq m(f)$ from the fact that $\eps<m(f)$. For the next claim, it suffices to consider $p<_f u<_f v<_f q$ by the same reason. In that case, we have
\begin{align*}
g(q)-g(p) &= f(q)-f(p)-\gamma \\
&= [f(q)-f(v)] + [f(v)-f(u) - \gamma] + [f(u)-f(p)] \\
&\geq m(f)d(v,q) + \left[ m(f) \left( 1+ \dfrac{\delta}{\|f\|} \right) - \dfrac{\gamma}{d(u,v)} \right] d(u,v) + m(f)d(p,u) \\
&\geq \left[ m(f) \left( 1+ \dfrac{\delta}{\|f\|} \right) - \dfrac{\gamma}{d(u,v)} \right] (d(p,u) + d(u,v) + d(v,q)) \\
&\geq \left[ m(f) \left( 1+ \dfrac{\delta}{\|f\|} \right) - \dfrac{\gamma}{d(u,v)} \right] d(p,q),
\end{align*}
where the second inequality follows from the inequality in \eqref{eq:smaller-than-m(f)}. This shows that $g$ strongly attains its minimum at $u \neq v \in M$.

(b)$\,\Longrightarrow\,$(a); Let $f$ be a bi-Lipschitz embedding from $M$ into $\R$. Since $\lam(f(M))>0$, we may fix an interval $I \subseteq \R$ such that $\lam(f(M) \cap I)>0$. By Proposition \ref{prop:subset-measure}, there exists a measurable $F \subseteq f(M) \cap I$ such that whenever an interval $J \subseteq I$ satisfies that $\lam(J \cap f(M))>0$, we have $0<\lam(J \cap F) < \lam(J \cap f(M))$. Define $\ph$ on $\R$ by
$$
\ph(t) = \begin{cases}
\,\dfrac{2\|f\|}{m(f)}\,, & \text{if } t \in F \text{ or } t \notin f(M), \\
\,\,1 \,,  & \text{if } t \in f(M) \setminus F,\phantom{\displaystyle \int_a^b}
\end{cases}
$$
and define $h \in \Lip(M)$ by $h(x) := \displaystyle \int_{f(0)}^{f(x)} \ph(t)\,dt$. Let $0<\eps<\|f\|/4$ be given, and let $g \in \SMA(M)$ be such that $\|g-h\|<\eps$. We will derive a contradiction.

First, by assumption, we have $\lam((f(M) \cap I) \setminus F)>0$. Let $t \in (f(M) \cap I) \setminus F$ be a Lebesgue density point of $(f(M) \cap I) \setminus F$. By Lemma \ref{lemma:densness-point}, for any $n \in \N$ there exist $a_n,b_n \in (f(M) \cap I) \setminus F$ such that $\frac{1}{n}(b_n-a_n) > \lam((a_n,b_n) \setminus (f(M) \setminus F))$. Thus if we write $p_n:=f^{-1}(a_n)$ and $q_n:=f^{-1}(b_n)$ for each $n \in \N$, we have
\begin{align*}
h(q_n)-h(p_n) &= \int_{a_n}^{b_n} \ph(t)\,dt \\
&\leq \lam((a_n,b_n) \setminus F) + \dfrac{2\|f\|}{m(f)} \cdot \lam ((a_n,b_n) \cap F) \\
&\leq \left(1+\dfrac{2\|f\|}{nm(f)}\right)(b_n-a_n) \\
&\leq \left(1+\dfrac{2\|f\|}{nm(f)}\right)d(p_n,q_n)\|f\|,
\end{align*}
where the second inequality follows from $(a_n,b_n) \cap F \subseteq (a_n,b_n) \setminus (f(M) \setminus F)$. Since $n \in \N$ was arbitrary, it follows that $m(h) \leq \|f\|$. Hence we can deduce from $\|g-h\|<\eps$ that $m(g) < \|f\| + \eps$.

Assume now that $\dfrac{g(v)-g(u)}{d(u,v)} = m(g)>0$ for some $u \neq v \in M$. Since
$$
h(q)-h(p) = \int_{f(p)}^{f(q)} \ph(t)\,dt \geq f(q)-f(p) \geq m(f)d(p,q)
$$
for every $p <_f q \in M$, we can deduce $m(h) \geq m(f)>0$. Moreover,
$$
\dfrac{h(v)-h(u)}{d(u,v)} \leq \dfrac{g(v)-g(u)}{d(u,v)} + \|g-h\| < \|f\| + 2\eps < \dfrac{3}{2} \|f\|.
$$
Now, we must have $\lam([f(u),f(v)] \cap f(M)) >0$; otherwise we have
$$
h(v)-h(u) = \int_{f(u)}^{f(v)} \ph(t)\,dt = \dfrac{2\|f\|}{m(f)} (f(v)-f(u)) \geq 2\|f\| d(u,v),
$$
which contradicts the fact that $\dfrac{h(v)-h(u)}{d(u,v)} < \|f\|$.

By the assumption of $F$, it follows that
$$
0< \lam([f(u),f(v)] \cap F) < \lam([f(u),f(v)] \cap f(M)).
$$
Let $t \in (f(u),f(v))$ be a Lebesgue density point of $F$. By Lemma \ref{lemma:densness-point}, there exists $c_n,d_n \in (f(u),f(v)) \cap F$ such that $\frac{1}{n}(d_n-c_n) > \lam((c_n,d_n) \setminus F)$. Writing $s_n := f^{-1}(c_n)$ and $t_n := f^{-1}(d_n)$, we have
\begin{align*}
h(t_n)-h(s_n) &= \int_{c_n}^{d_n} \ph(t)\,dt \\
&\geq \dfrac{2\|f\|}{m(f)} \cdot \lam ((c_n,d_n) \cap F) + \lam((c_n,d_n) \setminus F) \\
&\geq \dfrac{2\|f\|}{m(f)} \left( 1- \dfrac{1}{n}\right) (d_n-c_n) \\
&\geq 2\|f\| \left(1- \dfrac{1}{n}\right) d(s_n,t_n).
\end{align*}
Thus for sufficiently large $n_0 \in \N$, we may assume that
$$
\dfrac{h(t_{n_0})-h(s_{n_0})}{d(s_{n_0},t_{n_0})} \geq \dfrac{3}{2} \|f\|.
$$
However, since $f(u)<c_{n_0}<d_{n_0}<f(v)$, we also have $g$ strongly attains its minimum at $s_{n_0} \neq t_{n_0} \in M$ by Lemma \ref{lemma:minimum-preserve}. This is a contradiction since 
$$
\dfrac{h(t_{n_0})-h(s_{n_0})}{d(s_{n_0},t_{n_0})} \leq \dfrac{g(t_{n_0})-g(s_{n_0})}{d(s_{n_0},t_{n_0})} + \|g-h\| < \dfrac{3}{2}\|f\|.
$$

(b)$\,\Longleftrightarrow\,$(c); this equivalence is clear. Indeed, if $f\in\Lip(M)$ is a bi-Lipschitz embedding, then the bi-Lipschitz invariance of positivity of the one-dimensional Hausdorff measure yields
$$
\mathcal H^1(M)>0 \quad\Longleftrightarrow\quad \lambda(f(M))>0,
$$
which completes the proof.
\end{proof}

\section{Characterizations for vector-valued Lipschitz maps}\label{section:remarks}

In this section, we present results and observations concerning vector-valued Lipschitz maps and the denseness of strongly minimum-attaining Lipschitz maps, extending the results for scalar-valued Lipschitz functions. We will first provide a complete proof of Theorem \ref{theorem:RNP-characterization}. To this end, we show that the \textup{RNP} of the Lipschitz-free space over $M$ ensures the denseness of $\SMA(M,Y)$ regardless of the choice of $Y$.

\begin{lemma} \textup{(Stegall's optimization principle, \cite[Theorem 14]{S})}\label{lemma:Stegall}
Let $X$ be a Banach space, and let $D \subseteq X$ be a bounded \textup{RNP} set. Suppose that $\Phi : D \to [-\infty,\infty)$ is an upper semi-continuous and bounded above function. If $\Phi$ is not identically $-\infty$ on $D$, then the set
$$
\{x^* \in X^* : \Phi + \text{Re}\,x^* \text{ strongly exposes } D\}
$$
is a $G_\delta$-subset of $X^*$.
\end{lemma}

\begin{prop}\label{prop:RNP}
Let $M$ be a complete metric space such that $\F(M)$ has the \textup{RNP}. Then, $\SMA(M,Y)$ is dense in $\Lip(M,Y)$ for every Banach space $Y$.
\end{prop}

\begin{proof}
Let $f \in \Lip(M,Y)$ be given. Assume by Proposition \ref{prop:minimum-zero} that $m(f)>0$. Fix any $0<\eps<m(f)$. Since $\F(M)$ has the \textup{RNP}, $B_{\F(M)}$ is an \textup{RNP} set, where $B_{\F(M)}$ denotes the unit ball of $\F(M)$. Consider a function $\Phi : B_{\F(M)} \to [0,\infty]$ defined by
$$
\Phi(\mu) = \begin{cases}
\,\|T_f(\mu)\| \,, & \text{if } \mu \in \Mol(M)\\
\,+\infty \,,  & \text{if } \mu \notin \Mol(M).
\end{cases}
$$
Then, it is clear that $\Phi$ is lower semi-continuous and bounded below. Applying Lemma \ref{lemma:Stegall} with $B_{\F(M)}$ and $-\Phi$, there exists $\ph \in \F(M)^*$ with $\|\ph\|<\eps$ and $\mu_0 \in \Mol(M)$ such that the function
$$
\Psi(\mu) := \Phi(\mu) - \ph(\mu)
$$
satisfies that $\Psi(\mu_0) \leq \Psi(\mu)$ for every $\mu \in B_{\F(M)}$. One can derive easily that $\ph(\mu_0) \geq 0$ by symmetry of $\Phi$ and that $\Phi(\mu) \geq \ph(\mu)$ for any $\mu \in B_{\F(M)}$ since $m(f)>0$. Now, define a bounded linear operator $S \in \Lin(\F(M),Y)$ by
$$
S(\mu) := T_f(\mu) - \ph(\mu) \cdot \dfrac{T_f(\mu_0)}{\|T_f(\mu_0)\|} \quad \text{for } \mu \in \F(M).
$$
Here, $\|T_f(\mu_0)\| \neq 0$ since $m(f)>0$. We claim that $S$ attains its minimum at $\mu_0$. Indeed,
$$
\|S(\mu_0)\| = \left\| T_f(\mu_0) - \ph(\mu_0) \cdot \dfrac{T_f(\mu_0)}{\|T_f(\mu_0)\|} \right\| = \Psi(\mu_0).
$$
Also, for every $\mu \in B_{\F(M)}$, we have
\begin{align*}
\|S(\mu)\| &= \left\| T_f(\mu) - \ph(\mu) \cdot \dfrac{T_f(\mu_0)}{\|T_f(\mu_0)\|} \right\| \\
&\geq \|T_f(\mu)\| - |\ph(\mu)| \\
&= \|T_f(\alpha\mu)\| - \ph(\alpha u) \\
&= \Psi(\alpha \mu),
\end{align*}
for an appropriate sign $\alpha \in \{\pm 1 \}$. It is routine to see that
$$
\|S-T_f\| = \|\ph\| < \eps.
$$
Finally, consider a Lipschitz map $g \in \Lip(M,Y)$ which corresponds to $S$. Then, it is clear that $g \in \SMA(M,Y)$ from the correspondence between $\Lip(M,Y)$ and $\Lin(\F(M),Y)$, and we have $\|g-f\| = \|S-T_f\|<\eps$, as desired.
\end{proof}

It is noteworthy that Proposition \ref{prop:RNP} in fact generalizes the implication (a) $\Longrightarrow$ (b) in Theorem \ref{theorem:SMA(M)}, from the previous observation that $\mathcal{H}^1(M)>0$ is equivalent to $\F(M)$ failing the \textup{RNP} when $M$ is bi-Lipschitz equivalent to a subset of $\R$. We have chosen to retain the direct proof of this implication, since it is considerably more constructive and explicit. Combining the above result with Theorem \ref{theorem:SMA(M)}, we are now able to prove Theorem \ref{theorem:RNP-characterization}.

\begin{proof}[Proof of Theorem 2]
(a)$\Longrightarrow$(c) follows from Proposition \ref{prop:RNP}, since a bi-Lipschitz copy $N$ of a closed subset of $M$ is purely 1-unrectifiable if $\F(M)$ has the \textup{RNP}. (c)$\Longrightarrow$(b) is clear, and (b)$\Longrightarrow$(a) is a consequence of Theorem \ref{theorem:SMA(M)}.
\end{proof}

We now turn to the question of whether Theorem \ref{theorem:SMA(M)} extends to the vector-valued setting. Notice that several arguments in Section \ref{section:main-theorem} that rely heavily on the scalar-valued structure, such as those used in Lemma \ref{lemma:consecutive} and Lemma \ref{lemma:minimum-preserve}, do not extend directly to the vector-valued setting. This observation naturally raises the following question.

\begin{question}
Let $M$ be a complete metric space and $Y$ be a Banach space. Are the following statements equivalent?
\begin{enumerate}
\item[\textup{(a)}] $\SMA(M,Y)$ is not dense in $\Lip(M,Y)$.
\item[\textup{(b)}] $M$ admits a bi-Lipschitz embedding into $Y$ and $\mathcal{H}^1(M)>0$.
\end{enumerate}
\end{question}

Unfortunately, this equivalence fails in general in the vector-valued setting, as the next example shows. This example also serves as a counterexample to the analogous equivalence obtained by replacing the condition $\mathcal{H}^1(M)>0$ with the failure of the \textup{RNP} for $\F(M)$.

\begin{example}\label{example:Hilbert}
Let $\Gamma$ be any uncountable set. Then, $\SMA([0,1],\ell_2(\Gamma))$ is dense in $\Lip([0,1],\ell_2(\Gamma))$.
\end{example}

\begin{proof}
Let $\eps>0$ and $f \in \Lip([0,1],\ell_2(\Gamma))$ be given. Since the image $f([0,1])$ is separable, there exists a 2-dimensional subspace $H = \spann\{u_1,u_2\} \subseteq \ell_2(\Gamma)$ with orthonormal vectors $u_1,u_2 \in S_{\ell_2(\Gamma)}$ such that $H \subseteq \bigl(\overline{\spann}f([0,1]) \bigr)^\perp$. Fix $\delta>0$ such that
$$
\sqrt{\{m(f)\}^2 + \dfrac{4\eps^2}{\pi^2}} > m(f) + \delta,
$$
and choose $a<b$ in $[0,1]$ so that
$$
\dfrac{\|f(b)-f(a)\|}{b-a} < m(f) + \delta.
$$
Define $h \in \Lip([0,1],\ell_2(\Gamma))$ by
$$
h(t) := \dfrac{(b-a)u_1}{2\pi}\left(\cos \dfrac{2\pi(t-a)}{b-a}-1\right) + \dfrac{(b-a)u_2}{2\pi}\sin \dfrac{2\pi(t-a)}{b-a} \quad \text{for } t \in [0,1],
$$
and let $g := f+\eps h$. We claim that $g \in \SMA([0,1],\ell_2(\Gamma))$ with $\|g-f\| \leq \eps$. \\
First, by construction $h \in \Lip([0,1],\ell_2(\Gamma))$ with $\|h\| =1$, which implies that $\|g-f\| \leq \eps$. We also have that
$$
\| h(t)-h(s)\| \geq \dfrac{2}{\pi} |t-s|
$$
if $|t-s| \leq (b-a)/2$. Also, it is easy to see that $h(a) = h(b) = 0$. This gives that, if $0< |t-s| \leq (b-a)/2$, then
$$
\dfrac{\|g(t)-g(s)\|}{|t-s|} = \dfrac{\sqrt{\|f(t)-f(s)\|^2 + \eps^2 \|h(t)-h(s)\|^2}}{|t-s|} \geq \sqrt{\{m(f)\}^2 + \dfrac{4\eps^2}{\pi^2}} > m(f) + \delta
$$
from that $H \subseteq \bigl(\overline{\spann}f([0,1]) \bigr)^\perp$. Combining with
$$
\dfrac{\|g(b)-g(a)\|}{b-a} = \dfrac{\|f(b)-f(a)\|}{b-a} < m(f) + \delta,
$$
it follows that $g$ does not strongly attain its minimum at a pair with a distance less than or equal to $(b-a)/2$. By compactness of the set $\{(s,t) \in [0,1]^2: |s-t| \geq (b-a)/2\}$, we finally deduce that $g \in \SMA([0,1],\ell_2(\Gamma))$.
\end{proof}

On the other hand, one might naturally expect the equivalence to hold for instance after replacing $\mathcal{H}^1(M)$ with $\mathcal{H}^{\dim(Y)}(M)$. However, the next example shows that even this fails in general.

\begin{example}\label{example:ell_infty^n}
$\SMA([0,1],\ell_\infty^n)$ is not dense in $\Lip([0,1],\ell_\infty^n)$ for each $n \in \N$.
\end{example}

\begin{proof}
Fix any $n \in \N$. Let $f \in \Lip([0,1])$ be a Lipschitz function which cannot be approximated by functions in $\SMA([0,1])$, according to Theorem \ref{theorem:SMA(M)}. By Proposition \ref{prop:minimum-zero}, we have $m(f)>0$. Fix any $0<\eps<m(f)/2$. Define $\widehat{f} \in \Lip([0,1],\ell_\infty^n)$ by $\widehat{f}(t) := f(t)e_1$, where $e_1$ is the first canonical coordinate of $\ell_\infty^n$. Suppose, towards a contradiction, that $\widehat{f} \in \overline{\SMA([0,1]),\ell_\infty^n)}$. Then there exists $\widehat{g} \in \SMA([0,1],\ell_\infty^n)$ such that $\|\widehat{g}-\widehat{f}\|< \eps$. Define $g \in \Lip([0,1])$ by $g(t) := e_1^*(\widehat{g}(t))$, where $e_1^*$ is the first canonical coordinate of $(\ell_\infty^n)^* = \ell_1^n$. Since $\|\widehat{g}-\widehat{f}\|<m(f)/2$, we have for every $k = 2, \ldots, n$ that
$$
|e_1^*(\widehat{g}(t)-\widehat{g}(s))| \geq |e_k^*(\widehat{g}(t)-\widehat{g}(s))| \quad \text{for } s \neq t \in [0,1].
$$
Indeed, given $s \neq t \in [0,1]$, observe from $e_1^* \circ \widehat{f} = f$ that
\begin{align*}
|e_1^*(\widehat{g}(t)-\widehat{g}(s))| &\geq |f(t)-f(s)| - \bigl| e_1^*((\widehat{g}-\widehat{f})(t)-(\widehat{g}-\widehat{f})(s)) \bigr| \\
&\geq (m(f)-\|\widehat{g}-\widehat{f}\|)|t-s| \\
&\geq \eps|t-s| \\
&\geq |e_k^*((\widehat{g}-\widehat{f})(t)-(\widehat{g}-\widehat{f})(s))| \\
&= |e_k^*(\widehat{g}(t)-\widehat{g}(s))|.
\end{align*}
Thus it follows that
$$
\|\widehat{g}(t)-\widehat{g}(s)\| = |e_1^*(\widehat{g}(t)-\widehat{g}(s))| = |g(t)-g(s)| \quad \text{for } s \neq t \in [0,1],
$$
which shows that $m(g) = m(\widehat{g})$. As there exist $t_0 \neq s_0 \in [0,1]$ such that $\widehat{g}$ strongly attains its minimum at $t_0 \neq s_0$, it follows that $g$ strongly attains its minimum also at $t_0 \neq s_0 \in [0,1]$. Moreover, it is clear that
$$
\|g-f\| = |e_1^*(\widehat{g}-\widehat{f})| \leq \|\widehat{g}-\widehat{f}\| < \eps.
$$
Since $\eps>0$ was arbitrary, this contradicts that $f$ cannot be approximated by functions in $\SMA([0,1])$.
\end{proof}

To summarize, Theorem \ref{theorem:RNP-characterization} suggests that under a sufficiently strong condition on $M$, denseness can be guaranteed universally with respect to the range space. Meanwhile, Examples \ref{example:Hilbert} and \ref{example:ell_infty^n} show that the denseness of strongly minimum-attaining Lipschitz maps may depend heavily on the geometric aspect of range spaces. Motivated by the preceding discussion, we conclude the paper with the following general question.

\begin{question}
For which metric spaces $M$ and Banach spaces $Y$ is $\SMA(M,Y)$ dense in $\Lip(M,Y)$?
\end{question}

\vspace{5mm}

\noindent \textbf{Acknowledgments:\ } The author would like to thank Han Ju Lee for helpful discussions on the topic. The author was supported by the National Research Foundation of Korea(NRF) grant funded by the Korea government(MSIT) (RS-2026-25475373).

\end{document}